\documentclass{jgcc}

\pdfoutput=1
\usepackage{lastpage}
\jgccdoi{13}{1}{2}{7036}
\jgccheading{}{\pageref{LastPage}}{}{}%
{Jan.~02,~2021}{Mar.~24,~2021}{}


\usepackage{mathrsfs}
\usepackage{graphicx}

\usepackage{amssymb}
\usepackage{amsmath}
\usepackage{amsthm}
\usepackage{stmaryrd}

\usepackage{hyperref}
\theoremstyle{plain} 

\newcommand{\rk}{\mathrm{rk}}
\newcommand{\aut}{\mathrm{Aut}}

\newcommand{\bp}{\circledcirc}
\newcommand{\lab}{\operatorname{lab}}

\newcommand{\LL}{\mathcal{L}}
\newcommand{\FF}{\mathcal{F}}
\newcommand{\OO}{\mathcal{O}}
\newcommand{\AAEE}{\mathcal{AE}}
\newcommand{\CL}{\mathcal{C}l}
\newcommand{\OCL}{\mathcal{OC}l}
\newcommand{\FOCL}{f\mathcal{OC}l}

\newcommand{\leqfg}{\leqslant_{\textsf{fg}}}
\newcommand{\leqff}{\leqslant_{\textsf{ff}}}
\newcommand{\leqfi}{\leqslant_{\textsf{fi}}}
\newcommand{\leqalg}{\leqslant_{\textsf{alg}}}
\newcommand{\leqont}{\leqslant_{\textsf{ont}}}
\newcommand{\leqfont}{\leqslant_{\textsf{f.ont}}}

\begin{document}

\title{Onto extensions of free groups}

\author[S.~Mijares]{Sebasti\`a Mijares}	
\address{Departament d'Enginyeria Inform\`atica i de les Comunicacions, Universitat Aut\`onoma de Barcelona, CATALONIA}	
\email{sebastia.mijares@uab.cat} 

\author[E.~Ventura]{Enric Ventura}	
\address{Departament de Matem\`atiques, Universitat Polit\`ecnica de Catalunya, and Institut de Matem\`atiques de la UPC-BarcelonaTech, CATALONIA}	
\email{enric.ventura@upc.edu} 




\begin{abstract}
\noindent An extension of subgroups $H\leqslant K\leqslant F_A$ of the free group of rank $|A|=r\geqslant 2$ is called \emph{onto} when, for every ambient basis $A'$, the Stallings graph $\Gamma_{A'}(K)$ is a quotient of $\Gamma_{A'}(H)$. Algebraic extensions are onto and the converse implication was conjectured by Miasnikov--Ventura--Weil, and resolved in the negative, first by Parzanchevski--Puder for rank $r=2$, and recently by Kolodner for general rank. In this note we study properties of this new type of extension among free groups (as well as the fully onto variant), and investigate their corresponding closure operators. Interestingly, the natural attempt for a dual notion --\emph{into} extensions-- becomes trivial, making a Takahasi type theorem not possible in this setting.
\end{abstract}

\maketitle


The present paper is an elaborated version of Mijares~\cite{tfm}, the masters thesis defended by the first author at the Universitat Polit\`{e}cnica de Catalunya in July 2020. The goal of this masters thesis was to understand the recent interesting counterexample given by Kolodner~\cite{kolodner} to a previous conjecture from Miasnikov--Ventura--Weil~\cite{MVW}, and to investigate further the new notions of onto and fully onto extensions of free groups motivated by it.

The paper is organized as follows: in Section~1 we present the necessary context and background about free groups and Stallings graphs needed for the development; in Section~2 we reformulate the Miasnikov--Ventura--Weil conjecture and prove it, following an idea of Parzanchevski--Puder (who gave the first counterexample to the original version of the conjecture); in Section~3 we study the new concepts of onto and fully onto extensions of free groups arising from Kolodner's more elaborated counterexample to the strengthened Parzanchevski--Puder version of the conjecture; in Section~4, we study onto and fully onto closures of subgroups, and observe that the natural attempt for a dual notion (that of into entensions) becomes trivial, eliminating the possibility of a new version of Takahasi's theorem in this setting; finally, in Section~5, we state and comment some open questions arising naturally in this context.

All along the paper we write arguments on the left and homomorphisms on the right, $g\mapsto g\alpha$, and compositions accordingly, $g\mapsto g\alpha \mapsto (g\alpha)\beta=g\alpha\beta$.

\section{Context and background}

Let $A=\{a_1, \ldots ,a_r\}$ be an alphabet of $r$ letters, let $A^{\pm}=\{a_1, \ldots ,a_r, a_1^{-1}, \ldots ,a_r^{-1}\}$ be its formal involutive closure, and let $F_A$ be the free group on $A$ (formally, the free monoid on $A^{\pm}$ modulo the equivalence relation generated by elementary reduction, $a_ia_i^{-1}\sim a_i^{-1}a_i \sim 1$).

In 1983, elaborating on previous ideas by several authors, Stallings~\cite{Sta} established the notion of so-called \emph{Stallings $A$-automata}: oriented graphs (allowing loops and parallel edges) with labels from $A^{\pm}$ at the edges, being \emph{involutive} (i.e., for every edge $e$ from $p$ to $q$ with label $a$, there is another one $e^{-1}$ from $q$ to $p$ and labelled $a^{-1}$; $e$ and $e^{-1}$ are said to be inverse to each other), with a selected vertex called the \emph{basepoint} (denoted $\bp$), and being connected, \emph{deterministic} (no two different edges with the same label from, or into, the same vertex) and \emph{trim} (every vertex appears in some reduced closed path at the basepoint). Here, a path is called \emph{reduced} if it has no backtracking, i.e., no crossings of an edge immediately followed by its inverse (equivalently in the deterministic case, a path is reduced if and only if its label is a reduced word).

With this notion, Stallings~\cite{Sta} established a bijection between the set of (free) subgroups of $F_A$ and the set of isomorphism classes of \emph{Stallings $A$-automata}. Further, finitely generated subgroups correspond to finite Stallings $A$-automata and, when restricted to these subsets, the Stallings bijection is algorithmic friendly, i.e., there are fast algorithms for both directions: given a set of words $h_1,\ldots ,h_n$ on $A$, one can compute the Stallings graph $\Gamma_A(H)$ corresponding to the subgroup they generate, $H=\langle h_1, \ldots ,h_n \rangle\leqslant F_A$. And, given a Stallings $A$-automaton $\Gamma$, one can compute a basis for its \emph{language} subgroup $\LL(\Gamma)\leqslant F_A$ (the set of labels of reduced closed paths at the basepoint in $\Gamma$). We assume the reader is familiar with this theory; see~\cite{KM, MVW, Sta} for details.

\begin{thm}[Stallings, \cite{Sta}]
The following is a bijection:
 $$
\begin{array}{rcl} St \colon \{H\leqslant F_A\} & \longrightarrow & \{\mbox{isom. classes of Stallings $A$-automata}\} \\ H & \mapsto & \Gamma_{A}(H) \\ \LL(\Gamma) & \mapsfrom & \Gamma.\end{array}
 $$
Furthermore, $H\leqfg F_A$ if and only if $\Gamma_A(H)$ is finite; in this case, both directions are computable. \qed
\end{thm}

From now on, we will mostly consider involutive $A$-automata and will describe and draw them just mentioning their positive part (i.e., those edges labelled by letters in $A$); next to each positive edge we implicitly assume its inverse is also there (even if we do not mention it), ready to be used by paths around the automaton. The positive subautomaton of $\Gamma$ is denoted $\Gamma^+$ (and it is obviouly not involutive, unless edgeless).

For later use, we briefly explain how the above mentioned algorithms work. Suppose we are given a finite set of reduced words, $W=\{ h_1,\ldots ,h_n\}\subset F_A$. Draw the so-called \emph{flower automaton} $\FF(W)$: for each $i=1,\ldots ,n$, consider a circular graph with $A$-labels at the edges (usually called a \emph{petal}) in such a way that, when travelled around, it spells the word $h_i=h_i(a_1, \ldots ,a_r)$ (or its inverse if travelled in the opposite direction), and glue all of them together along their basepoints $\bp$. Note that $\FF(W)$ is trim, deterministic except maybe at the basepoint, and satisfies $\LL(\FF(W))=H$, where $H=\langle W\rangle \leqslant F_A$. To obtain a deterministic automaton, successively apply elementary foldings: whenever we have two edges with the same label from, or into, the same vertex, identify them into a single one (as well as their terminal, or initial, vertices). This process always terminates in a finite number of steps, and it can be proven that the final $A$-automaton obtained in this way, denoted $\Gamma_A(H)$, is deterministic, trim, reads the same language $\LL(\Gamma_A(H))=\LL(\FF(W))=H$, and, more significantly, is independent from the specific sequence of foldings applied, and even from the set of generators of $H$ we started with: the Stallings graph $\Gamma_A(H)$ only depends on the subgroup $H\leqslant F_A$ (and on the ambient basis $A$ chosen to work with). Conversely, given a finite Stallings $A$-automaton $\Gamma$, we can choose a spanning tree $T$ (more precisely, the involutive closure of a spanning tree of $\Gamma^+$), and get the basis of $H=\LL(\Gamma)\leqslant F_A$ given by $\{h_e \mid e\in E\Gamma^+\setminus ET\}$, where $h_e=\lab(T[\bp, \iota e]eT[\tau e, \bp])$ and $T[p,q]$ stands for the unique reduced path in $T$ from vertex $p$ to vertex $q$, i.e., the \emph{$T$-geodesic} from $p$ to $q$.


By \emph{degree} of a vertex $p$ in an involutive $A$-automaton $\Gamma$ we mean the out-degree of $p$. Note that this equals the in-degree of $p$ in $\Gamma$, and also the total (in- plus out-) degree of $p$ in its positive part $\Gamma^+$.

For an $A$-automaton $\Gamma$, define its \emph{core}, denoted $c(\Gamma)$, as its largest trim subautomaton, i.e., that determined by the vertices and edges appearing in some reduced closed path at $\bp$; so, $\Gamma$ is trim if and only if $c(\Gamma)=\Gamma$. It is easy to see that, when $\Gamma$ is finite and connected, this is equivalent to saying that no vertex in $\Gamma$ has degree 1 except maybe $\bp$; in this case, one can get $c(\Gamma)$ from $\Gamma$ by applying finitely many times the trim operation: remove a vertex of degree one different from $\bp$ (together with the corresponding edge).

An \emph{$A$-automata homomorphism} (\emph{$A$-homomorphism}, for short) from $\Gamma$ to $\Gamma'$ is a pair of maps $\theta=(\theta_V, \theta_E)$, $\theta_V\colon V\Gamma \to V\Gamma'$ and $\theta_E\colon E\Gamma \to E\Gamma'$ (where $V\Gamma$ and $E\Gamma$ denote the sets of vertices and edges of $\Gamma$, respectively), such that $\bp\theta =\bp'$ and, for every $a$-labelled edge $e$ from $p$ to $q$ in $\Gamma$, $e\theta_E$ is an $a$-labelled edge from $p\theta_V$ to $q\theta_V$ in $\Gamma'$; we shall abuse language and write $\theta=\theta_V=\theta_E$. Such a $\theta$ is called \emph{onto} (resp. \emph{into}) if both $\theta_V$ and $\theta_E$ are onto (resp. into). For an onto $A$-homomorphism we will use the notation $\theta\colon \Gamma \twoheadrightarrow \Gamma'$. Observe that, for $\Gamma$ connected and $\Gamma'$ deterministic, there exists at most one $A$-homomorphism from $\Gamma$ to $\Gamma'$.

Stallings bijection behaves well with respect to inclusions in the sense that, for two subgroups $H,K\leqslant F_A$, $H\leqslant K$ if and only if there is an $A$-homomorphism from $\Gamma_A(H)$ to $\Gamma_A(K)$ which, in this case, is unique and will be denoted $\theta_{H,K}\colon \Gamma_A(H) \to \Gamma_A(K)$.

Since 1983, the Stallings bijection became central for the modern understanding and study of the lattice of subgroups of a free group. With the development of these graphical techniques, many new results have been obtained about free groups and their subgroups. Also, most of the results known before Stallings~\cite{Sta} have been reproven using graphical techniques, usually with conceptually simpler and more transparent proofs. Takahasi theorem is a typical example illustrating this fact.

An extension of subgroups $H\leqslant K\leqslant F_A$ is called \emph{free} (we also say that $H$ is a \emph{free factor} of $K$), denoted $H\leqff K$, whenever some (so, any) basis of $H$ can be extended to a basis for $K$; for example, it is easy to see that, for $H\leqslant K\leqslant F_A$, if $\Gamma_A(H)$ is a subautomaton of $\Gamma_A(K)$ then $H\leqff K$ (just extend a spanning tree for $\Gamma_A(H)$ to a spanning tree for $\Gamma_A(K)$), while the converse is far from true. The notion of free factor is the non-abelian version of the notion of direct summand from commutative algebra. In a vector space, every pair of subspaces $E\leqslant F$ is in direct sum position, $E\leqslant_{\oplus} F$, i.e., $F=E\oplus E'$ for some complementary subspace $E'$. When we consider, for example, free abelian groups (i.e., free modules over $\mathbb{Z}$) the exact same result is not true, but it works if we admit a bit of flexibility: every subgroup $H\leqslant \mathbb{Z}^m$ is of finite index only in finitely many subgroups $H=H_0\leqfi H_1, \ldots ,H_n\leqslant \mathbb{Z}^m$ and, for every $K\leqslant \mathbb{Z}^m$ containing $H$, there exists a unique $i$ such that $H\leqfi H_i\leqslant_{\oplus} K\leqslant \mathbb{Z}^m$. Of course, the situation in the free group seems much wilder, starting from the well known fact that $H\leqslant K\leqslant F_A$ does not even imply $\rk(H)\leqslant \rk(K)$ (to the extreme that $F_{\aleph_0}$ can be viewed as a subgroup of $F_2$). However, back in the 1950's, Takahasi~\cite{Tak} proved that, again, the same result adapts to the free group case, after admitting a little bit more of degeneration: we will have to restrict ourselves to finitely generated subgroups, and we will lose the finite index condition.

Takahasi theorem was proved 70 years ago using purely combinatorial and algebraic techniques. However, in more recent years, it was rediscovered independently, by Ventura~\cite{V} in 1997, by Margolis--Sapir--Weil~\cite{MSW} in 2001, and by Kapovich--Miasnikov~\cite{KM} in 2002, in slightly different contexts; see also the subsequent paper Miasnikov--Ventura--Weil~\cite{MVW} joining the three points of view. These authors, independently, gave their own proofs of Takahasi's theorem, and they happened to be essentially the same proof; we would say, the ``natural" proof of this result using Stallings graphs. Let us sketch it here, since it will play a central role along the rest of the paper.

\begin{thm}[Takahasi, \cite{Tak}]\label{Taka}
Let $H\leqfg F_A$. Then $H$ determines finitely many finitely generated extensions $H=H_0, H_1, \ldots ,H_n\leqslant F_A$ such that, for every $K\leqslant F_A$ containing $H$, there exists $i$ such that $H\leqslant H_i\leqff K\leqslant F_A$.
\end{thm}

\begin{proof}[Sketch of proof]
Given $H\leqfg F_A$, consider its (finite) Stallings graph $\Gamma_A(H)$. Identifying certain sets of vertices, we get a new $A$-automaton which may very well be not deterministic; apply then a sequence of foldings until obtaining a deterministic one, say $\Gamma_1$, which will correspond to a finitely generated overgroup of $H$, say $H\leqslant \LL(\Gamma_1)\leqslant F_A$. Repeating this operation for every possible partition on the set of vertices $V\Gamma_A(H)$ (and possibly getting the same result for different initial partitions), we obtain a finite number of overgroups of $H$, say $\OO_A (H)=\{H=H_0, H_1, \ldots ,H_n\}$, called the \emph{$A$-fringe of $H$} in~\cite{V}, which satisfies the statement of Takahasi theorem. In fact, let $K$ be a (non-necessarily finitely generated) subgroup with $H\leqslant K\leqslant F_A$, and look at the corresponding $A$-homomorphism $\theta_{H,K}\colon \Gamma_A(H)\to \Gamma_A(K)$. Looking at the image $\operatorname{Im} (\theta_{H,K})$ as a (finite and deterministic) subautomaton of $\Gamma_A(K)$, we see that: (1) the $A$-homomorphism $\theta_{H,K}\colon \Gamma_A(H)\twoheadrightarrow \operatorname{Im}(\theta_{H,K})$ is onto and so, $\LL(\operatorname{Im}(\theta_{H,K}))\in \OO_A(H)$; and (2) $\operatorname{Im}(\theta_{H,K})$ is a subautomaton of $\Gamma_A(K)$ and so, $\LL(\operatorname{Im}(\theta_{H,K}))\leqff K$. This proves that the $A$-fringe of $H$ is a finite family of finitely generated overgroups of $H$ satisfying the property stated in Takahasi theorem (a \emph{Takahasi family}, for short).
\end{proof}

As done in Kapovich--Miasnikov~\cite{KM} and in Miasnikov--Ventura--Weil~\cite{MVW}, it is possible to clean up the \emph{fringe of $H$}, $\OO_A(H)$, in order to obtain a minimal Takahasi family, and gain uniqueness of the middle subgroup $H_i$ in Theorem~\ref{Taka}; moreover, this minimal Takahasi family will have a clear algebraic meaning as follows. A subgroup extension $H\leqslant K\leqslant F_A$ is called \emph{algebraic}, denoted $H\leqalg K$, if $H$ is not contained in any proper free factor of $K$; denote by $\AAEE(H)$ the set of algebraic extensions of $H$ (within $F_A$). By Takahasi's theorem, $\AAEE(H)\subseteq \OO_A(H)$ and so, $|\AAEE(H)|<\infty$ for every finitely generated $H$. Furthermore it can be seen that, applying the following cleaning process to $\OO_A(H)$, one obtains precisely $\AAEE(H)$: \emph{for each pair of distinct subgroups $H_i, H_j\in \OO_A(H)$, if $H_i\leqff H_j$ then delete $H_j$ from the list}; see~\cite{MVW} for details.

Using the language of algebraic extensions, one can deduce the following variant (and slight improvement) of Takahasi's theorem:

\begin{thm}[{\cite[Thm.~2.6, Prop.~3.9, Thm.~3.16]{MVW}}]\label{alg cl}
Let $H\leqfg F_A$. The set of algebraic extensions $\AAEE(H)$ is finite and computable; further, it satisfies that, for every $H\leqfg K\leqslant F_A$, there exists a unique intermediate subgroup $L$ such that $H\leqalg L\leqff K\leqslant F_A$. This $L$ is called the \emph{$K$-algebraic closure of $H$}, denoted $L=\CL_K(H)$, and coincides both with the smallest free factor of $K$ containing $H$, and with the largest algebraic extension of $H$ contained in $K$. \qed
\end{thm}

In Theorem~\ref{alg cl}, the words \emph{smallest} and \emph{largest} make sense because of the following basic properties of free and algebraic extensions:

\begin{prop}\label{intersectionfreefactors}
Let $H_i \leqslant K_i \leqslant F_A$ be a collection of subgroup extensions in $F_A$, $i \in I$. Then,
\begin{enumerate}[(i)]
\item (\cite[Lem.~2.4]{MVW}) $H_i \leqff K_i$, $\forall i\in I$ $\Rightarrow$ $\cap_{i\in I} H_i \leqff \cap_{i\in I} K_i$;
\item (\cite[Prop.~3.12]{MVW}) $H_i \leqalg K_i$, $\forall i\in I$ $\Rightarrow$ $\langle H_i, i\in I\rangle \leqalg \langle K_i, i\in I\rangle$. \qed
\end{enumerate}
\end{prop}

For later use, we collect some more properties of free and algebraic extensions, highlighting the duality of these two notions.

\begin{prop}[{\cite[Prop.~3.11]{MVW}}]\label{compositionfreefactors}
Let $H\leqslant M_i\leqslant K\leqslant F_A$, for $i=1,2$. Then,
\begin{enumerate}[(i)]
\item if $H\leqalg M_1 \leqalg K$, then $H\leqalg K$;
\item[(i$'$)] if $H\leqff M_1 \leqff K$, then $H\leqff K$;
\item  if $H\leqalg K$, then $M_1\leqalg K$, while $H\not \leqalg M_1$ in general;
\item[(ii$'$)] if $H\leqff K$, then $H\leqff M_1$, while $M_1\not\leqff K$ in general;
\item if $H\leqalg M_1$ and $H\leqalg M_2$, then $H\leqalg \langle M_1\cup M_2 \rangle$, while $H\not\leqalg M_1\cap M_2$ in general;
\item[(iii$'$)] if $H\leqff M_1$ and $H\leqff M_2$, then $H\leqff M_1\cap M_2$, while $H\not \leqff \langle M_{1} \cup M_{2} \rangle$ in general. \qed
\end{enumerate}
\end{prop}

Let us insist on the computability part in Theorem~\ref{alg cl}. If we start with a finitely generated subgroup $H\leqfg F_A$ (given by a finite set of generators), we can effectively compute: (1) its Stallings graph $\Gamma_A(H)$; (2) all its (finitely many) quotients, resulting from identifying vertices in $\Gamma_A(H)$ in all possible ways followed by folding (i.e., we can effectively compute bases for all the members of the fringe $\OO_A(H)$); and (3) the cleaning process until getting the set of (bases of all the) algebraic extensions of $H$. This last step requires an algorithm deciding whether a given extension $H\leqslant K$ is free or not; this can be done using the classical Whitehead techniques (see Roig--Ventura--Weil~\cite{RVW} for a significant improvement on the time complexity, from exponential to polynomial), or using more modern algorithms based on Stallings automata (see Silva--Weil~\cite{SW}, and Puder~\cite{P}). Therefore, for $H\leqfg F_A$, the set $\AAEE(H)$ is finite and computable.

Another important observation is the following. The fringe of $H$ strongly depends on the ambient basis $A$ (reflected in the notation with the subscript $A$ in $\OO_A(H)$), while the set of algebraic extensions $\AAEE(H)$ does not, and is canonically associated to the subgroup $H$, since it is defined completely in algebraic terms. To illustrate this fact, see Example~2.5 from Miasnikov--Ventura--Weil~\cite{MVW}, where the fringe of $H=\langle ab, acba\rangle \leqslant F_A$, $A=\{a,b,c\}$, is computed: $\OO_A(H)=\{ H_0, H_1, H_2, H_3, H_4, H_5\}$, where $H_0=H$, $H_1=\langle ab, ac, ba \rangle$, $H_2=\langle ba, ba^{-1}, cb\rangle$, $H_3=\langle ab, ac, ab^{-1}, a^2 \rangle$, $H_4=\langle ab, aca, acba\rangle$, and $H_5=\langle a, b, c\rangle =F_A$; however, with respect to the new ambient basis $A'=\{ d, e, f\}$, where $d=a$, $e=ab$, and $f=acba$, the $A'$-automaton $\Gamma_{A'}(H)$ has a single vertex, and hence the $A'$-fringe of $H$ is much simpler, $\OO_{A'}(H)=\{H\}$. Of course, in this example, $H\leqff F_A$ and $\AAEE(H)=\{H\}$. Alternatively, thinking the change of basis as an automorphism of the ambient free group, we can express the above fact by saying that, for every $\varphi\in \aut(F_A)$, $\AAEE(H\varphi)=\{ K\varphi \mid K\in \AAEE(H)\}$ (this is to say that $H\leqalg K$ if and only if $H\varphi\leqalg K\varphi$), while $\OO_A(H)$ and $\OO_A(H\varphi)$ are unrelated in general (they may even have different cardinals). In the example above, considering the automorphism $\varphi\colon F_A\to F_A$, $a\mapsto a$, $b\mapsto ab$, $c\mapsto acba$, we have $H\varphi^{-1}=\langle b,c\rangle$, $\AAEE(H)=\{H\}$, $\AAEE(H\varphi^{-1})=\{H\varphi^{-1}\}$, $|\OO_A(H)|=6$, and $|\OO_A(H\varphi^{-1})|=1$.

We interpret the above fact by thinking that $\AAEE(H)$ is what really carries relevant algebraic information about the subgroup $H$ and its relative position within the lattice of subgroups of $F_A$. And $\OO_A(H)$ is the same set locally distorted with some accidental new members depending on the ambient basis used to draw and work with the graphs. From this point of view, Miasnikov--Ventura--Weil launched in~\cite{MVW} the following natural conjecture: the common subgroups in $\OO_{A'}(H)$, when $A'$ runs over every ambient basis might be, precisely, the algebraic extensions:

\begin{conj}[Miasnikov--Ventura--Weil, \cite{MVW}]\label{conjecturaMVW}
Let $A$ be a finite alphabet, and $F_A$ be the free group on $A$. Then, for every $H\leqfg F_A$,
 $$
\AAEE(H) =\bigcap_{A'\,\, \mbox{\scriptsize basis of } F_A} \OO_{A'}(H) =\bigcap_{\varphi\in \aut(F_A)} \big( \OO_A(H\varphi) \big)\varphi^{-1}.
 $$
\end{conj}

In~\cite{MVW} it was mentioned that this is clearly true in the two extremal situations $H\leqfi F_A$ and $H\leqff F_A$, but nothing else was known at that time. Seven years later, in 2014, the paper Parzanchevski--Puder~\cite{ppuder} appeared showing that the conjecture is not true as stated:

\begin{prop}[{Parzanchevski--Puder, \cite[Prop.~4.1]{ppuder}}]\label{pp}
Let $A=\{a,b\}$ and consider the free group $F_A$. The extension $H=\langle a^2b^2\rangle\leqslant \langle a^2b^2, ab\rangle =K\leqslant F_A$ is free, $H\leqff K$ (so, it is not algebraic), but it satisfies $K\in \OO_{A'}(H)$ for every ambient basis $A'$. \qed
\end{prop}

They also proposed two possible reformulations making the conjecture more plausible. On one hand, the authors recognize that their counterexample exploits many idiosyncrasies of the (automorphism group of the) free group of rank two, and it could be that Miasnikov--Ventura--Weil conjecture holds true for ambient free groups of rank three or more (i.e., $\aut(F_2)$ is much ``smaller" and easier in structure than $\aut(F_r)$ for $r\geqslant 3$ and so, the intersection of fringes with respect to all ambient basis is ``too lax" in the case of rank two). On the other hand, they made the elementary but clever observation that the independence of $\AAEE(H)$ from the ambient basis (i.e., the reason for the obvious inclusion $\AAEE(H)\subseteq \bigcap_{A'} \OO_{A'}(H)$) has an even more restrictive consequence: adding new letters, extend $A$ to a bigger (possibly infinite) new alphabet $A\subseteq B$, and consider the free extension $F_A \leqff F_B$; viewed as subgroups of $F_B$, it is still true that $H\leqalg K$ if and only if $H\varphi\leqalg K\varphi$, for every $\varphi\in \aut(F_B)$. So, the same reasoning gives us the stronger inclusion $\AAEE(H)\subseteq \bigcap_{B'} \OO_{B'}(H)$, where $B'$ runs now over \emph{all} the ambient bases of \emph{all free extensions} $F_B$, $B\supseteq A$. Parzanchevski and Puder finished their paper~\cite{ppuder} by reformulating Conjecture~\ref{conjecturaMVW} into the following two variations:

\begin{conj}[{Parzanchevski--Puder, \cite{ppuder}}]\label{conj-a}
Let $A$ be a finite alphabet with $|A|\geqslant 3$, and $F_A$ be the free group on $A$. Then, for every $H\leqfg F_A$,
 $$
\AAEE(H) =\bigcap_{A'\,\, \mbox{\scriptsize basis of } F_A} \OO_{A'}(H) =\bigcap_{\varphi\in \aut(F_A)} \Big( \OO_A(H\varphi) \Big) \varphi^{-1}.
 $$
\end{conj}

\begin{conj}[{Parzanchevski--Puder, \cite[Conj.~5.1]{ppuder}}]\label{conj-b}
Let $A$ be a finite alphabet and $F_A$ be the free group on $A$. Then, for every $H\leqfg F_A$,
 $$
\AAEE(H) =\bigcap_{B\supseteq A}\, \bigcap_{B'\,\, \mbox{\scriptsize basis of } F_B} \OO_{B'}(H) =\bigcap_{B\supseteq A} \, \bigcap_{\varphi\in \aut(F_B)} \Big( \OO_{B}(H\varphi) \Big) \varphi^{-1}.
 $$
\end{conj}

They also observed that their counterexample to Conjecture~\ref{conjecturaMVW} does not serve as a counterexample for either Conjecture~\ref{conj-a} or Conjecture~\ref{conj-b}. In fact, $H=\langle a^2b^2\rangle$ and $K = \langle a^2b^2, ab\rangle$ live inside the free group of rank two $F_A$, with $A=\{a,b\}$; but if we extend this ambient free group with a third letter, say $B=\{a,b,c\}$, then the new ambient basis $B'=\{x,y,z\}$, with $x=a$, $y=cb^{-1}$, $z=cbc^{-1}$, breaks the counterexample since $H=\langle x^2y^{-1}z^2y\rangle$, $K=\langle x^2y^{-1}z^2y, xy^{-1}zy\rangle$ and $K\not\in \OO_{B'}(H)$. We shall further exploit this example below.

The last step in this story is the recent 2020 preprint Kolodner~\cite{kolodner}, where the author definitively disproves the conjecture in all its mentioned forms. In fact, he shows the following stronger result:

\begin{thm}[{Kolodner~\cite[Thm.~1.4]{kolodner}}]\label{kolodner}
Let $A=\{a,b\}$. In $F_A$, the proper subgroup extension $H=\langle b^2aba^{-1}\rangle \leqff \langle b, aba^{-1}\rangle=K$ is free (and so, not algebraic) but, for an arbitrary alphabet $B$, and for every homomorphism $\varphi\colon F_A\to F_B$ with $a\varphi, b\varphi \neq 1$, $K\varphi \in \OO_{B}(H\varphi)$. \qed
\end{thm}

\section{Closing the MVW conjecture}

In this section, we want to elaborate more on Parzanchevski--Puder's idea about possible natural modifications of the original Miasnikov--Ventura--Weil conjecture, which could make it true.

Let $H\leqslant K\leqslant F_A$.

Firstly, we can look above $F_A$, not just through free extensions but using \emph{all} possible extensions. That is, consider a new free group $F_B$ and an arbitrary injective homomorphism $F_A\hookrightarrow F_B$, not necessarily with $A\subseteq B$, i.e., with the image not necessarily being a free factor of $F_B$. In this situation, $H\leqalg K$ still implies that, for every $\varphi\in \aut(F_B)$, $H\varphi\leqalg K\varphi \leqslant F_B$; or, in other words, $H\leqalg K$ implies that, for every basis $B'$ of $F_B$, the $B'$-homomorphism $\theta_{H,K}\colon \Gamma_{B'}(H)\twoheadrightarrow \Gamma_{B'}(K)$ is onto.

Secondly, we can look downwards instead of upwards: $H\leqalg K\leqslant F_A$ also implies that, for every subgroup $K\leqslant L\leqslant F_A$ and every automorphism $\varphi\in \aut(L)$, $H\varphi\leqalg K\varphi \leqslant L\leqslant F_A$; or, in other words, $H\leqalg K$ implies that, for every $K\leqslant L\leqslant F_A$ and every basis $C'$ of $L$, the $C'$-homomorphism $\theta_{H,K}\colon \Gamma_{C'}(H)\twoheadrightarrow \Gamma_{C'}(K)$ is onto. Note that, in general, there are plenty of automorphisms of $L$ which do not extend to automorphisms of $F_A$; furthermore, $L$ may very well be not finitely generated.

Thirdly, we can combine the two effects upwards/downwards: $H\leqalg K$ also implies that, for every free group inclusion $F_A \hookrightarrow F_B$, every subgroup $L\leqslant F_B$ containing $K$, and every automorphism $\varphi\in \aut(L)$, $H\varphi\leqalg K\varphi \leqslant L\leqslant F_B$ (note that this is more general than before since $L$ is now not necessarily a subgroup of $F_A$); equivalently, $H\leqalg K$ implies that, for every free group inclusion $F_A \hookrightarrow F_B$, every subgroup $L\leqslant F_B$ containing $K$, and every basis $C'$ of $L$, the $C'$-homomorphism $\theta_{H,K}\colon \Gamma_{C'}(H)\twoheadrightarrow \Gamma_{C'}(K)$ is still onto.

As we show in the following result, all these generalizations of Parzanchevski--Puder's idea are, in fact, (tautologically) equivalent to the algebraicity of the initial extension $H\leqslant K$.

\begin{prop}
Let $A$ be an alphabet, $F_A$ be the free group on $A$, and let $H\leqslant K\leqslant F_A$. Then the following are equivalent:
\begin{enumerate}[(a)]
\item $H\leqalg K$;
\item for every free group inclusion $F_A \hookrightarrow F_B$, every subgroup $L\leqslant F_B$ containing $K$, and every basis $C'$ of $L$, the $C'$-homomorphism $\theta_{H,K}\colon \Gamma_{C'}(H)\twoheadrightarrow \Gamma_{C'}(K)$ is onto;
\item for every $B\supseteq A$, every subgroup $L\leqslant F_B$ containing $K$, and every basis $C'$ of $L$, the $C'$-homomorphism $\theta_{H,K}\colon \Gamma_{C'}(H) \twoheadrightarrow \Gamma_{C'}(K)$ is onto;
\item for every subgroup $L\leqslant F_A$ containing $K$, and every basis $C'$ of $L$, the $C'$-homomorphism $\theta_{H,K}\colon \Gamma_{C'}(H)\twoheadrightarrow \Gamma_{C'}(K)$ is onto;
\item for every basis $C'$ of $K$, the $C'$-homomorphism $\theta_{H,K}\colon \Gamma_{C'}(H)\twoheadrightarrow \Gamma_{C'}(K)$ is onto.
\end{enumerate}
\end{prop}

\begin{proof}
The implications (a)$\Rightarrow$(b)$\Rightarrow$(c)$\Rightarrow$(d)$\Rightarrow$(e) are obvious.

For the relevant one, (e)$\Rightarrow$(a), consider a free decomposition $K=K_1*K_2$ with $H\leqslant K_1$; applying the hypothesis to a basis of $K$ of the form $C'=C'_1\sqcup C'_2$, where $C'_i$ is a basis of $K_i$, $i=1,2$, we have that $\theta_{H,K}\colon \Gamma_{C'}(H)\twoheadrightarrow \Gamma_{C'}(K)$ is onto. But, by construction, $\Gamma_{C'}(H)$ contains only edges labelled by letters from $C'_1$, while $\Gamma_{C'}(K)$ has a single vertex, and edges in bijection with $C'$. Therefore, $C'_2$ must be empty and $K_2=1$. This proves that $H\leqalg K$.
\end{proof}

\section{Onto extensions}

Interestingly, the counterexample given by Kolodner in Theorem~\ref{kolodner} opens up a possible new line of research considering and studying two new types of subgroup extensions within the lattice of subgroups of a free group (which do not coincide, in general, with algebraic extensions).

\begin{defi}
Let $A$ be an alphabet, and let $H\leqslant K\leqslant F_A$. We say that this is an \emph{onto} extension of subgroups, denoted $H\leqont K$, if $\theta_{H,K}\colon \Gamma_{A'}(H)\to \Gamma_{A'}(K)$ is onto, for every basis $A'$ of $F_A$; in other words, if $K\in \bigcap_{A'} \OO_{A'}(H)$, where $A'$ runs over all possible basis for $F_A$. Further, we say that $H\leqslant K$ is \emph{fully onto}, denoted $H\leqfont K$, if $\theta_{H,K}\colon \Gamma_{B'}(H)\to \Gamma_{B'}(K)$ is onto, for every basis $B'$ of every free extension $F_A\leqff F_B$, $B\supseteq A$. We denote by $\Omega(H)$ (resp., $f\Omega(H)$) the set of onto (resp., fully onto) extensions of $H$ within $F_A$.
\end{defi}

\begin{prop}
Let $H\leqslant K\leqslant F_A$. The following implications are true, while the reverse implications are not true, in general:
 $$
H\leqalg K \,\,\, \Rightarrow \,\,\, H\leqfont K \,\,\, \Rightarrow \,\,\, H\leqont K \,\,\, \Rightarrow \,\,\, K\in \OO_A(H).
 $$
That is, $\AAEE(H)\subseteq f\Omega(H)\subseteq \Omega(H)\subseteq \OO_{A}(H)$.
\end{prop}

\begin{proof}
We have already seen the three implications in the previous section. As for counterexamples, take $A=\{a,b\}$, and observe that Kolodner example $H=\langle b^2aba^{-1}\rangle \leqslant \langle b, aba^{-1}\rangle=K\leqslant F_A$ satisfies $H\leqfont K$ but $H\not\leqalg K$ (in fact, it is the extreme opposite, $H\leqff K$). Parzanchevski--Puder example $H=\langle a^2b^2\rangle\leqslant \langle a^2b^2, ab\rangle=K\leqslant F_A$ satisfies $H\leqont K$ but $H\not\leqfont K$. Finally, in Example~2.5 from Miasnikov--Ventura--Weil~\cite{MVW}, we have $H=\langle ab, acba\rangle$ and $H_1=\langle ab, ac, ba\rangle\in \OO_{A}(H)$, while $H\not\leqont H_1$ (see the discussion above).
\end{proof}

\begin{cor}
For any finitely generated subgroup $H\leqfg F_A$, we have $|\AAEE(H)|\leqslant |f\Omega(H)|\leqslant |\Omega(H)|\leqslant |\OO_A(H)|<\infty$.\qed
\end{cor}

Let us investigate now the properties of these two new types of extensions among free groups. To do this, we need to use an idea, which is not explicitly written in Kolodner~\cite{kolodner} but it is reminiscent in the arguments there.

Let $\Gamma_0$ and $\Delta_0$ be connected $A$-automata (neither necessarily deterministic, nor trim) and let $\theta_0\colon \Gamma_0\to \Delta_0$ be an $A$-homomorphism; let $H=\LL(\Gamma_0)$ and $K=\LL(\Delta_0)$. We want to fold and trim both $\Gamma_0$ and $\Delta_0$ until obtaining the Stallings graphs $\Gamma_A(H)$ and $\Gamma_A(K)$, respectively, but in a synchronized way so that $\theta_0$ keeps inducing $A$-homomorphisms down the tower of foldings, until $\theta_{H,K}\colon \Gamma_A(H)\to \Gamma_A(K)$; further, we shall pay attention to the preservation of surjectivity, whenever possible. Here is a way of doing this (used in Kolodner~\cite{kolodner} to analyze the relation between the automata homomorphisms $\theta_{H,K}\colon \Gamma_{A}(H)\to \Gamma_{A}(K)$ and $\theta_{H\varphi ,K\varphi }\colon \Gamma_{B}(H\varphi)\to \Gamma_{B}(K\varphi)$, for $\varphi \colon F(A)\to F(B)$):
\begin{enumerate}
\item[(0)] Start with $\theta_0\colon \Gamma_0\to \Delta_0$; let $H=\LL(\Gamma_0)$ and $K=\LL(\Delta_0)$.
\item[(1)] For every pair of edges, $e_1, e_2$, violating determinism in $\Gamma_0$, fold them in $\Gamma_0$ and simultaneously fold their images $e_1\theta_0$ and $e_2\theta_0$ in $\Delta_0$; there is the possibility that $e_1\theta_0$ and $e_2\theta_0$ are already equal in $\Delta_0$, in which case we do nothing on the right hand side. Observe that after this (or these) folding operation(s), the $A$-homomorphism $\theta_0$ determines naturally an $A$-homomorphism among the resulting $A$-automata. Repeat this process until having no more foldings to do at the left hand side, and denote the result by $\theta_1\colon \Gamma_1\to \Delta_1$. By construction, $\Gamma_1$ is deterministic. Note also that if $\theta_0$ is onto then $\theta_1$ is also onto.

\item[(2)] Now perform all possible foldings remaining to be done at the right hand side (and nothing on the left hand side): the homomorphism $\theta_1$ naturally transfers to the new situation, and we get $\theta_2\colon \Gamma_2\to \Delta_2$, where $\Gamma_2=\Gamma_1$, and now both $\Gamma_2$ and $\Delta_2$ are deterministic (and not trim in general). Note that, again, if $\theta_1$ is onto then $\theta_2$ is also onto.
\item[(3)] At this point, observe that the edges in the core of $\Gamma_2$ must map through $\theta_2$ to edges in the core of $\Delta_2$ (this is because both $\Gamma_2$ and $\Delta_2$ are deterministic and $\theta_2$ is an $A$-homomorphism). So, edges outside the core of $\Delta_2$ can only be images of edges from outside the core of $\Gamma_2$. Hence, trimming all the edges from outside the core of $\Delta_2$, and trimming simultaneously all their $\theta_2$-preimages in $\Gamma_2$, we obtain $\theta_3\colon \Gamma_3\to \Delta_3$, where $\Delta_3=c(\Delta_2)$ is deterministic and trim, and $\Gamma_3$ is deterministic (and not yet trim, in general). Again, $\theta_2$ onto implies $\theta_3$ onto. Moreover, observe also that $\LL(\Delta_3)=\LL(\Delta_2)= \LL(\Delta_1)=\LL(\Delta_0)=K$ and so, $\Delta_3=\Gamma_{A}(K)$.
\item[(4)] Finally, let us finish trimming $\Gamma_3$ (and do nothing on the right hand side) to obtain $\theta_4\colon \Gamma_4\to \Delta_4$, where $\Delta_4=\Delta_3= \Gamma_{A}(K)$, $\Gamma_4=\Gamma_{A}(H)$ since it is already a Stallings $A$-automaton with $\LL(\Gamma_4)=H$, $\theta_4=\theta_{H, K}$, and we are done. It is crucial to note, however, that in this \emph{critical} last step we may very well lose surjectivity: in fact, even with $\theta_3$ being onto, removing edges from $\Gamma_3$ may result into some edges from $\Delta_4=\Delta_3$ having no $\theta_4$-preimages in $\Gamma_4$.
\end{enumerate}

This synchronized folding process is crucial for the proof of the next proposition. Observe that Propositions~\ref{a} and~\ref{b} express the fact that onto and fully onto extensions satisfy the same properties we already know for algebraic extensions so, they behave very similarly to them (compare with Propositions~\ref{intersectionfreefactors} and~\ref{compositionfreefactors}).

\begin{prop}\label{a}
Let $H_i \leqslant K_i \leqslant F_A$ be a collection of subgroup extensions in $F_A$, $i \in I$. Then,
\begin{enumerate}[(i)]
\item $H_i \leqont K_i$, $\forall i\in I$ $\Rightarrow$ $\langle H_i, i\in I\rangle \leqont \langle K_i, i\in I\rangle$;
\item $H_i \leqfont K_i$, $\forall i\in I$ $\Rightarrow$ $\langle H_i, i\in I\rangle \leqfont \langle K_i, i\in I\rangle$.
\end{enumerate}
\end{prop}

\begin{proof}
(i). Fix a basis $A'$ of $F(A)$. The hypothesis tells us that, for $i\in I$, $\theta_{H_i,K_i}\colon \Gamma_{A'}(H_i)\twoheadrightarrow \Gamma_{A'}(K_i)$ is an onto $A'$-homomorphism. Glue together all the $\Gamma_{A'}(H_i)$ (resp., $\Gamma_{A'}(K_i)$), for $i\in I$, along their basepoints, to get $\Gamma_0$ (resp., $\Delta_0$) and the natural onto $A'$-homomorphism $\theta_0\colon \Gamma_0 \twoheadrightarrow \Delta_0$ induced by the $\theta_{H_i,K_i}$'s. Observe that neither $\Gamma_0$ nor $\Delta_0$ is deterministic, in general, and that $\LL(\Gamma_0)=H$ and $\LL(\Delta_0)=K$, where $H=\langle H_i, i\in I\rangle$ and $K=\langle K_i, i\in I\rangle$.

Now let us apply the synchronized folding process described above, starting with the onto $A'$-homomorphism $\theta_0\colon \Gamma_0\twoheadrightarrow \Delta_0$, until obtaining $\theta_{H,K}\colon \Gamma_{A'}(H)\to \Gamma_{A'}(K)$. We will deduce that this last $A'$-homomorphism is onto, after arguing that all the left hand side $A'$-automata $\Gamma_0$, $\Gamma_1$, $\Gamma_2$, along the process (as well as the right hand side ones) are trim and so, neither step (3) nor the critical step (4) take place.

By construction, $\theta_0\colon \Gamma_0\twoheadrightarrow \Delta_0$ is onto, and $\Gamma_0$ and $\Delta_0$ are both trim, and not necessarily deterministic. In general, along the individual folding processes applied to $\Gamma_0$ and $\Delta_0$ in steps (1) and (2), the fact of being trim can be lost, since some vertices along the process decrease their degrees and could eventually become degree one vertices. However, we claim that this will not be the case, neither for $\Gamma_0$, nor for $\Delta_0$. Observe that these two $A'$-automata are trim in a stronger way: for every vertex $p\neq \circledcirc$, not only the degree is bigger than one, $|\{e\in E \mid \iota e =p \}|>1$, but also its \emph{label-degree}, $|\{a\in (A')^{\pm} \mid \exists\, e\in E \mbox{ s.t. } \iota e=p,\, \operatorname{lab}(e)=a\}|>1$. And, meanwhile the degree of a vertex could decrease along the folding process, its label-degree stays constant or increases. Therefore, at the end of step (2), the label-degree of all vertices $p\neq \circledcirc$ in $\Gamma_2$ and $\Delta_2$ are bigger than one and hence, so are their degrees too. This means that steps (3) and (4) are empty and $\Gamma_2=\Gamma_3=\Gamma_4=\Gamma_{A'}(H)$, $\Delta_2=\Delta_3=\Delta_4=\Gamma_{A'}(K)$, and $\theta_2=\theta_3=\theta_4=\theta_{H,K}\colon \Gamma_{A'}(H)\twoheadrightarrow \Gamma_{A'}(K)$ is onto. Since this is valid for every initially fixed basis $A'$ of $F(A)$, we deduce that $H\leqont K$.

(ii). Fix a free extension $F_A\leqff F_B$, $B\supseteq A$, and a basis $B'$ for $F_B$. The hypothesis tells us that, each $B'$-homomorphism $\theta_{H_i,K_i}\colon \Gamma_{B'}(H_i)\twoheadrightarrow \Gamma_{B'}(K_i)$, for $i\in I$, is onto. Glue them together along their basepoints and apply the synchronized folding process described above. The exact same argument as in (i) tells us that $\theta_{H,K}\colon \Gamma_{B'}(H) \twoheadrightarrow \Gamma_{B'}(K)$ is onto, where $H=\langle H_i, i\in I\rangle$ and $K=\langle K_i, i\in I\rangle$. Since this is valid for every basis $B'$, we deduce that $H\leqfont K$.
\end{proof}

\begin{rem}
The argument in the proof of Proposition~\ref{a} is not technically correct for the case $|I|=\infty$, since the processes of folding edges, pair by pair, in steps~(1) and~(2) could very well be infinitely long. This is not a conceptual obstacle, but only a matter of expression: one should do \emph{all} these foldings in a single step (losing, of course, the algorithmic aspect of the proof, valid only when $|I|<\infty$). In a non-deterministic (possibly infinite) $A$-automata $\Gamma$, one can define the equivalence relation among vertices given by $p\sim q$ $\Leftrightarrow$ there is a path $\gamma$ in $\Gamma$ satisfying $\iota \gamma=p$, $\tau \gamma=q$, and $\lab(\gamma )=1\in F_A$. It is straightforward to see that $\Gamma/\sim$ is automatically deterministic and has the same language $\LL(\Gamma/\sim)=\LL(\Gamma)$; further, for the case when $\Gamma$ is finite, $\Gamma/\sim$ equals the final output of the sequence of foldings. Similarly, when $\Gamma$ is infinite, the trim process cannot be done edge by edge (there could even be no vertex of degree 1, and infinitely many edges to be trimmed out). Instead, one should delete, in a single step, all the edges not visited by any reduced closed path at $\circledcirc$.
\end{rem}

\begin{prop}\label{b}
Let $H\leqslant M_i\leqslant K\leqslant F_A$, for $i=1,2$. Then,
\begin{enumerate}[(i)]
\item if $H\leqont M_1 \leqont K$, then $H\leqont K$;
\item if $H\leqont K$, then $M_1\leqont K$, while $H\not \leqont M_1$ in general;
\item if $H\leqont M_1$ and $H\leqont M_2$, then $H\leqont \langle M_1\cup M_2 \rangle$, while $H\not\leqont M_1\cap M_2$ in general;
\item[(i$'$)] if $H\leqfont M_1 \leqfont K$, then $H\leqfont K$;
\item[(ii$'$)] if $H\leqfont K$, then $M_1\leqfont K$, while $H\not\leqfont M_1$ in general;
\item[(iii$'$)] if $H\leqfont M_1$ and $H\leqfont M_2$, then $H\leqfont \langle M_1\cup M_2 \rangle$, while $H\not \leqfont M_1\cap M_2$ in general.
\end{enumerate}
\end{prop}

\begin{proof}
(i)-(i'). Let $A'$ be a basis for $F_A$. Since, by hypothesis, both $A'$-homomor\-phisms $\theta_{H,M_1}\colon \Gamma_{A'}(H)\twoheadrightarrow \Gamma_{A'}(M_1)$ and $\theta_{M_1,K}\colon \Gamma_{A'}(M_1 )\twoheadrightarrow \Gamma_{A'}(K)$ are onto, we have $\theta_{H,K}=\theta_{H,M_1}\circ \theta_{M_1,K}\colon \Gamma_{A'}(H)\twoheadrightarrow \Gamma_{A'}(M_1 ) \twoheadrightarrow \Gamma_{A'}(K)$ is onto as well. Hence, $H\leqont K$. The proof of (i') is analogous.

(ii)-(ii'). Let $A'$ be a basis for $F_A$. Since, by hypothesis, the $A'$-homomor\-phism $\theta_{H,K}\colon \Gamma_{A'}(H)\twoheadrightarrow \Gamma_{A'}(K)$ is onto, and the inclusions $H\leqslant M_1\leqslant K$ tell us that $\theta_{H,K}=\theta_{H,M_1}\circ \theta_{M_1,K}$, we deduce that $\theta_{M_1, K}\colon \Gamma_{A'}(M_1 ) \twoheadrightarrow \Gamma_{A'}(K)$ is onto as well. Hence, $M_1\leqont K$. The proof of (ii') is analogous.

For a counterexample to the other assertion, consider $\langle a^2b^2\rangle \leqalg \langle a,b\rangle=F_{\{a,b\}}$, which is an algebraic extension because $a^2b^2$ is neither a proper power, nor a primitive element in $\langle a,b\rangle$; consequently, $\langle a^2b^2\rangle \leqfont \langle a,b\rangle$ and $\langle a^2b^2\rangle \leqont \langle a,b\rangle$ but, clearly, $\langle a^2b^2\rangle \not\leqont \langle a^2b^2, a^3b^3\rangle \leqslant \langle a,b\rangle$.

(iii)-(iii'). This is a direct application of Proposition~\ref{a}.

For a counterexample to the other assertion, consider $F_{\{a,b\}}$ and its subgroups $H=\langle a^6b^6\rangle$, $M_1=\langle a^2, b^2\rangle$, and $M_2=\langle a^3, b^3\rangle$. Since $a^6b^6=(a^2)^3(b^2)^3$ is neither a proper power nor a primitive element in $M_1$, we have $H\leqalg M_1$ and hence, $H\leqfont M_1$ and $H\leqont M_1$; similarly, $H\leqfont M_2$ and $H\leqont M_2$. However, we claim that $H\not\leqont M_1\cap M_2$. It is easier to show first that $H\not\leqfont M_1\cap M_2$, and then recycle the idea to strengthen the result to $H\not\leqont M_1\cap M_2$. In fact, $M_1\cap M_2 =\langle a^6, b^6\rangle$ (where, by the way, the element $a^6b^6$ is now primitive); adding a third ambient letter $\{a,b\}\subseteq \{a,b,c\}$, and using the ambient basis $B'=\{x,y,z\}$ where $x=ac^{-1}$, $y=cb$, $z=c$, it is straightforward to see that $\theta_{H, M_1\cap M_2}\colon \Gamma_{B'}(H)\to \Gamma_{B'}(M_1\cap M_2)$ is not onto: the idea here is that, in this new basis, $a=xz$, $b=z^{-1}y$ and so, in the middle point of the petal $a^6b^6=(xz)^5(x\not{z})(\not{z^{-1}}y)(z^{-1}y)^5$, an edge labelled $z$ must be trimmed out provoking the lost of surjectivity. A bit trickier to find, but following the same idea, one can see that, in the ambient basis $A'=\{x,y\}$ with $x=ab^{-1}a^{-1}$ and $y=ab^2$, we have $a=x^2y$, $b=y^{-1}x^{-1}y$, and the $A'$-homomorphism $\theta_{H, M_1\cap M_2}\colon \Gamma_{A'}(H)\to \Gamma_{A'}(M_1\cap M_2)$ is not onto. Therefore, $H\not\leqont M_1\cap M_2$ as claimed.
\end{proof}

The first examples of a fully onto extension not being algebraic (given by Kolodner in
Theorem~\ref{kolodner}, namely $\langle b^2aba^{-1}\rangle \leqfont \langle b, aba^{-1}\rangle\leqslant F_{\{a,b\}}$) or of an onto extension not being fully onto (given by Parzanchevski--Puder in Proposition~\ref{pp}, namely $\langle a^2b^2\rangle\leqont \langle a^2b^2, ab\rangle \leqslant F_{\{a,b\}}$) were hard to establish. However, using the above properties, and combining with known examples of algebraic extensions, we can easily construct lots of new examples of onto and fully onto extensions.

\begin{exa}
Take two copies of Kolodner's example, $\langle b^2aba^{-1}\rangle \leqfont \langle b, aba^{-1}\rangle\leqslant F_{\{a,b\}}$ and $\langle c^2aca^{-1}\rangle \leqfont \langle c, aca^{-1}\rangle\leqslant F_{\{a,c\}}$; clearly, both are also examples of fully onto extensions of subgroups of $F_{\{a,b,c\}}$ so, by Proposition~\ref{a}(ii), we obtain $\langle b^2aba^{-1}, c^2aca^{-1}\rangle \leqfont \langle b, c, aba^{-1}, aca^{-1}\rangle\leqslant F_{\{a,b,c\}}$, while it is a free extension so it is not algebraic. However, we have to be careful with a similar attempt to use Proposition~\ref{a}(i): $\langle a^2b^2\rangle\leqont \langle a^2b^2, ab\rangle$ as subgroups of $F_{\{a,b\}}$, and $\langle a^2c^2\rangle\leqont \langle a^2c^2, ac\rangle$ as subgroups of $F_{\{a,c\}}$, but we cannot conclude that $\langle a^2b^2, a^2c^2\rangle\leqont \langle a^2b^2, ab, a^2c^2, ac\rangle$ as subgroups of $F_{\{a,b,c\}}$ since, in the ambient basis $A'=\{x,y,z\}$, with $x=a$, $y=cb^{-1}$, $z=cbc^{-1}$, the corresponding $A'$-homomorphism is not onto (the problem being that \emph{neither} of the two initial extensions is onto when viewed as subgroups of $F_{\{a,b,c\}}$).

As a second example, consider $\langle a^2b^2\rangle \leqont \langle a^2b^2, ab\rangle$ and $\langle a^{-2}b^{-2}a^2b^2\rangle\leqont \langle a^2, b^2\rangle$, both as subgroups of $F_{\{a,b\}}$ (the first one is Parzanchevski--Puder example, and the second one is algebraic). Then, by Proposition~\ref{a}(i), $\langle a^2b^2, b^2a^2\rangle\leqont \langle ab, a^2, b^2\rangle$, while $\langle a^2b^2, b^2a^2\rangle\not\leqfont \langle ab, a^2, b^2\rangle$ since, extending with a third letter $c$ and using the ambient basis $B'=\{x,y,z\}$, with $x=a$, $y=cb^{-1}$, $z=cbc^{-1}$, the corresponding $B'$-homomorphism is not onto.
\end{exa}

\begin{cor}
For $H\leqfg F_A$, the finite sets $\AAEE(H)$, $f\Omega(H)$, and $\Omega(H)$, partially ordered by natural inclusion, form three lattices with the join operations given by $H_1 \vee H_2 =\langle H_1 \cup H_2\rangle$, and the meet operations given by
 $$
H_1 \wedge_{\AAEE} H_2 =\left\langle K \in \AAEE(H) \mid K\leqslant H_1 \cap H_2 \right\rangle,
 $$
 $$
H_1 \wedge_{f\Omega} H_2 =\left\langle K \in f\Omega(H) \mid K\leqslant H_1 \cap H_2 \right\rangle,
 $$
 $$
H_1 \wedge_{\Omega} H_2 =\left\langle K \in \Omega(H) \mid K\leqslant H_1 \cap H_2 \right\rangle,
 $$
respectively. Moreover, these lattices are not semimodular, in general (and so, not distributive either).
\end{cor}

\begin{proof}
By Propositions~\ref{compositionfreefactors}(iii), \ref{b}(iii'), and ~\ref{b}(iii), $H\leqslant \langle H_1 \cup H_2\rangle$ is algebraic (resp., fully onto, onto) whenever $H\leqslant H_1$ and $H\leqslant H_2$ are so; and clearly it is the smallest such extension containing both $H_1$ and $H_2$, hence it is the join in the three lattices, $H_1\vee_{\AAEE} H_2=H_1\vee_{f\Omega} H_2=H_1\vee_{\Omega} H_2=\langle H_1 \cup H_2\rangle$. By a similar reason, the subgroup generated by all the $K\in \AAEE(H)$ (resp., $K\in f\Omega(H)$, $K\in \Omega(H)$) contained in $H_1\cap H_2$ (among which we always have $H$ itself) is algebraic (resp., fully onto, onto) over $H$, and it is clearly the biggest such extension contained in $H_1\cap H_2$, hence it is the meet $H_1\wedge_{\AAEE} H_2$ (resp., $H_1\wedge_{f\Omega} H_2$, $H_1\wedge_{\Omega} H_2$).

A lattice is 
\emph{semimodular} when $H_1 \wedge H_2 <:H_1$ implies $H_2<: H_1\vee H_2$ (here, $H_1<:H_2$ means that $H_2$ \emph{covers} $H_1$, i.e., $H_1<H_2$ but there is no $H_3$ in between $H_1<H_3<H_2$). This property is not true, in general, for $\AAEE(H)$, neither $\Omega(H)$, nor $f\Omega(H)$, as the example computed in~\cite[Ex.~3.7]{tfm} shows. Consider $H=\langle a^2, b^2, ab^2a \rangle\leqslant F_A$, where $A=\{a,b\}$. It is straightforward to see that $\OO_A(H)=\{H_0, H_1, H_2, H_3, H_4, H_5, H_6 \}$, where $H_0=H$, $H_1=\langle a^2, b, a^{-1}b^2a\rangle$, $H_2=\langle a^2, b^2, a^{-1}ba\rangle$, $H_3=\langle a^2, b^2, ab\rangle$, $H_4=\langle a, b^2\rangle$, $H_5=\langle a^2, b, a^{-1}ba\rangle$, and $H_6=\langle a,b\rangle$. It is also straightforward to see that all these extensions are algebraic over $H$ so, $\AAEE(H)=f\Omega(H)=\Omega(H)=\OO_A(H)$. Since all the inclusions among these subgroups follow by transitivity from $H_0\leqslant H_1, H_2, H_3, H_4$,\,\, $H_1, H_2\leqslant H_5$, and $H_5, H_3, H_4\leqslant H_6$, we have that 
$H_3$ covers $H_2 \wedge H_3 = H_0$, but $H_2 \vee H_3 = H_6$ does not cover $H_2$ (since $H_2<H_5<H_6$). Hence, neither of the three lattices $\AAEE(H)$, $f\Omega(H)$, and $\Omega(H)$ (here coinciding) is semimodular, in general. Since distributivity implies semimodularity, they are not distributive either.
\end{proof}

\section{The onto and fully onto closures and triviality of into extensions}

Similarly to what is done for algebraic extensions, given an extension $H\leqslant K\leqslant F_A$, one can define the onto and the fully onto closures of $H$ relative to $K$.

\begin{defi}
Let $H\leqslant K\leqslant F_A$. The \emph{$K$-onto closure} of $H$, denoted $\OCL_K(H)$, is the join of all the onto extensions of $H$ contained in $K$, i.e., the biggest onto extension of $H$ inside $K$. Similarly, the \emph{$K$-fully onto closure} of $H$, denoted $\FOCL_K(H)$, is the join of all the fully onto extensions of $H$ contained in $K$, i.e., the biggest fully onto extension of $H$ inside $K$. Clearly, by construction, we have the inclusions $H\leqslant \CL_K(H)\leqslant \FOCL_K(H)\leqslant \OCL_K(H)\leqslant K$.
\end{defi}

\begin{rem}
Since, in general, the notions of algebraic, onto, and fully onto extensions do not coincide, the corresponding closure operators will also be different. In fact, we can illustrate this fact by recycling Kolodner's and Parzanchevski--Puder's examples above. Let $A=\{a,b\}$. We know that $H=\langle b^2aba^{-1}\rangle$ and $K=\langle b, aba^{-1}\rangle$ satisfy $H\leqfont K\leqslant F_A$, while $H\leqff K$. Hence, $\CL_K(H)=H$, $\FOCL_K(H)=K$, and $\OCL_K(H)=K$. On the other hand, we also know that $H=\langle a^2b^2\rangle$ and $K=\langle a^2b^2, ab\rangle$ satisfy $H\leqont K\leqslant F_A$ but $H\not\leqfont K$. It is straightforward to compute $\OO_A(H)=\{H_0, H_1, H_2, H_3, H_4, H_5, H_6\}$, where $H_0=H$, $H_1=\langle a, b^2\rangle$, $H_2=\langle a^2, b\rangle$, $H_3=\langle a^2, b^2\rangle$, $H_4=\langle ab, a^2b^2\rangle$, $H_5=\langle a^2, b^2, ab\rangle$, and $H_6=\langle a,b\rangle=F_A$. Furthermore, $\AAEE(H)=\{H_0, H_6\}$ so, $\CL_K(H)=H$. But, additionally, the only subgroups in $\OO_A(H)$ between $H$ and $K$ are $H_0=H$ and $H_4=K$ so, $\FOCL_K(H)=H$ and $\OCL_K(H)=K$.
\end{rem}

\begin{prop}
Let $H\leqfg K\leqslant F_A$ be an extension of subgroups. Then, we have
 $$
H\leqalg \CL_{K}(H) \!\begin{array}{l} \leqslant_{_{\textsf{f.ont}}} \\ \leqslant_{_{\textsf{ff}}} \end{array}\!\! \FOCL_{K}(H) \!\begin{array}{l} \leqslant_{_{\textsf{ont}}} \\ \leqslant_{_{\textsf{ff}}} \end{array}\!\! \OCL_{K}(H)\leqff K.
 $$
\end{prop}

\begin{proof}
The chain of inclusions is clear by construction. Also, by construction, we have
$H\leqalg \CL_{K}(H)$, $H\leqfont \FOCL_{K}(H)$ and $H\leqont \OCL_{K}(H)$ therefore, by Proposition~\ref{b}(ii)(ii'), $\CL_{K}(H)\leqfont \FOCL_{K}(H)$ and $\FOCL_{K}(H) \leqont \OCL_{K}(H)$.

To see the free factors, observe that $\CL_{K}(H)\leqff K$ and so, by Proposition~\ref{compositionfreefactors}(ii'), $\CL_K(H)\leqff \FOCL_K(H)$ and $\CL_K(H)\leqff \OCL_K(H)$. The remaining two free factors come from transitivity of onto and fully onto extensions (see Proposition~\ref{b}(i)(i')) and from applying Theorem~\ref{alg cl} to the corresponding extensions: there exists a subgroup $L$ such that $\FOCL_{K}(H) \leqalg L \leqff \OCL_{K}(H)$ so, $H\leqfont \FOCL_{K}(H) \leqfont L\leqslant K$ hence, $H\leqfont L\leqslant K$ and, by maximality, $\FOCL_{K}(H)=L\leqff \OCL_{K}(H)$. Similarly,
there exists a subgroup $M$ such that $\OCL_{K}(H) \leqalg M \leqff K$ so, $H\leqont \OCL_{K}(H) \leqont M\leqslant K$ hence, $H\leqont M\leqslant K$ and, again by maximality, $\OCL_{K}(H)=M\leqff K$.
\end{proof}

\begin{cor}
The three closures of a subgroup $H\leqfg F_A$ with respect to the ambient free group $F_A$ do coincide: $\CL_{F_A}(H)=\FOCL_{F_A}(H)=\OCL_{F_A}(H)\leqff F_A$.
\end{cor}

\begin{proof}
It is enough to see the inclusion $\OCL_{F_A}(H) \leqslant \CL_{F_A}(H)$. In fact, $\CL_{F_A}(H)\leqff F_A$ and we consider a basis $\{a'_1, \ldots ,a'_r\}$ for $\CL_{F_A}(H)$ and an extension of it to a basis $A'=\{a'_1, \ldots ,a'_r, a'_{r+1}, \ldots ,a'_n\}$ of $F_A$. By construction $H\leqont \OCL_{F_A}(H)$ so, in particular, the $A'$-homomorphism $\theta_{H, \OCL_{F_A}(H)}\colon \Gamma_{A'}(H)\twoheadrightarrow \Gamma_{A'}(\OCL_{F_A}(H))$ is onto. But $H\leqslant \CL_{F_A}(H)=\langle a'_1, \ldots ,a'_r\rangle$ so $\Gamma_{A'}(H)$, and hence $\Gamma_{A'}(\OCL_{F_A}(H))$, has no edges labelled by $a'_{r+1}, \ldots ,a'_n$. Therefore, $\OCL_{F_A}(H)\leqslant \langle a'_1, \ldots ,a'_r\rangle =\CL_{F_A}(H)$.
\end{proof}

\medskip

Arriving at this point, it seems natural to ask which could be the possible dual notions to the concepts of onto and fully onto, in the same way that free extensions are dual to algebraic extensions. Fulfilling this duality, Theorem~\ref{alg cl} states that the $K$-algebraic closure of $H\leqslant K$ can alternatively be reached by taking the biggest algebraic extension of $H$ contained in $K$, or the smallest free factor of $K$ containing $H$. Is there a dual notion for onto (resp., fully onto) extensions in this sense? i.e., is there a property ${\mathcal P}$ of extensions for which $\OCL_K(H)$ (resp., $\FOCL_K(H)$) is the smallest ${\mathcal P}$-subgroup of $K$ containing $H$?

Since being onto is less restrictive than being algebraic, the dual notion should be a more restrictive notion than being a free factor. We presume the reader can easily imagine a very natural candidate for this possible dual notion: we could say that an extension $H\leqslant K\leqslant F_A$ is \emph{into} if $\theta_{H,K}\colon \Gamma_{A'}(H)\hookrightarrow \Gamma_{A'}(K)$ is injective for every basis $A'$ of $F_A$. And similarly, for \emph{fully into}.

However, somehow against intuition, this dualization project fails dramatically. After studying these two concepts, and proving some promising properties (very similar to those of free factors), we realized that the only into extension $H\leqslant K\leqslant F_A$ is the equality $H=K$, by first proving that $H\leqslant_{\textsf{into}} K$ implies that $\Gamma_{A'}(H)$ must be a \emph{full} subgraph of $\Gamma_{A'}(K)$ for every ambient basis $A'$, and then seeing that this situation forces equality $H=K$. So, these notions of into and fully into are trivial and, by no means, can they provide a Theorem similar to~\ref{alg cl}.

At the time of writing the present paper, a subsequent preprint Kolodner~\cite{kolodner2} appeared providing an easier proof for this same fact. We redirect the reader there.

\begin{prop}[{Kolodner~\cite[Thm.~21]{kolodner2}}]
For every extension of finitely generated subgroups $H\leqfg K\leqfg F_A$, there always exists a basis $A'$ of the ambient group $F_A$ such that the $A'$-homomorphism $\theta_{H,K}\colon \Gamma_{A'}(H)\twoheadrightarrow \Gamma_{A'}(K)$ is onto. \qed
\end{prop}

\begin{cor}[{Kolodner~\cite[Cor.~22]{kolodner2}}]
Let $H\leqfg K\leqfg F_A$. If the $A'$-homomor\-phism $\theta_{H,K}\colon \Gamma_{A'}\hookrightarrow \Gamma_{A'}(K)$ is injective for every basis $A'$ of $F_A$ then, $H=K$. \qed
\end{cor}

\section{Open questions}

We conclude with a list of interesting related questions, which remain open as far as we know. 

\begin{qu}
\emph{Is there an algorithm to decide whether a given extension $H\leqfg K\leqfg F_A$ is onto? and fully onto? Is there an algorithm to compute onto and fully onto closures of given extensions $H\leqslant K$ of finitely generated subgroups?}
\end{qu}

\begin{rem}
In the situation $H\leqfg K\leqfg F_A$, we can compute $\OO_{A}(H)$ and keep all those overgroups of $H$ contained in $K$. This provides a finite list of subgroups containing both $\OCL_K(H)$ and $\FOCL_K(H)$. However, to finish deciding who they are among the candidates in the list, we would need an algorithm deciding whether a given extension $H\leqfg L\leqfg F_A$ is onto (resp., fully onto) or not. A procedure for this was designed by Kolodner in~\cite{kolodner}, which may stop and give the correct answer, or may work forever (the example in Theorem~\ref{kolodner} was obtained, precisely, as a stopping instance for this procedure). To make it into a true algorithm we would need a proof that it always stops, or an additional criterium to kill the process at certain point, and get the answer in finitely many steps.
\end{rem}

\begin{qu}
\emph{What is the algebraic meaning of an extension being onto, or fully onto? Is it possible to characterize the facts $H\leqont K\leqslant F_A$ and $H\leqfont K\leqslant F_A$ without refereing to the bases of $F_A$?}
\end{qu}

\begin{rem}
These two definitions are canonical in the sense that they do not depend on any prefixed ambient basis. It would be interesting, then, to characterize them in algebraic terms with respect to $H$ and $K$, but not talking about ambient bases (like the definition of algebraic extension). It seems, however, that free factors will probably not help in this, since there exist onto, and even fully onto, extensions being simultaneously free factors. A tricky detail to take into account here is that fully onto transfers through free extensions while onto does not, i.e., $H\leqfont K$ as subgroups of $F_A$ implies $H\leqfont K$ as subgroups of $F_B$ for any $B\supseteq A$, while the same is not true in general for onto extensions.
\end{rem}

\begin{qu}
\emph{Is there a notion dual to onto, or dual to fully onto, which may lead to a Theorem similar to~\ref{alg cl}?
}\end{qu}

\section*{Acknowledgements}

The first named author acknowledges support from the Spanish Agencia Estatal de Investigaci\'on within the grants for PhD candidate researchers FPI 2019 (AEI/FSE, EU). The second named author acknowledge partial support from the Spanish Agencia Estatal de Investigaci\'on, through grant MTM2017-82740-P (AEI/ FEDER, UE), and also from the Graduate School of Mathematics through the ``Mar\'{\i}a de Maeztu" Programme for Units of Excellence in R\&D (MDM-2014-0445).

\end{document}